\documentclass[12pt, reqno]{article}
\usepackage{enumerate,amsmath,amssymb,bm,ascmac,amsthm,url}
\usepackage{tikz}
\usepackage[margin=30truemm]{geometry}
\usepackage{here}

\theoremstyle{definition}
\newtheorem{thm}{Theorem}[section]
\newtheorem{thm*}{Theorem}

\newtheorem{defi*}{Definition}
\newtheorem{lem}[thm]{Lemma}
\newtheorem{lem*}{Lemma}
\newtheorem{pro}[thm]{Proposition}
\newtheorem{pro*}{Proposition}

\newtheorem{question}[thm]{Question}

\newcommand{\MC}[1]{\mathcal{#1}}
\newcommand{\MB}[1]{\mathbb{#1}}

\newcommand{\BM}[1]{{\bm #1}}

\newcommand{\MM}[4]{\begin{bmatrix} #1 & #2 \\ #3 & #4 \end{bmatrix}}

\newcommand{\G}{\Gamma}

\newcommand{\D}{\Delta}

\newcommand{\sikaku}{\mathrel{\square}}

\DeclareMathOperator{\Spec}{Spec}

\DeclareMathOperator{\tr}{tr}

\DeclareMathOperator{\diag}{diag}
\DeclareMathOperator{\Ker}{Ker}

\makeatletter

\@addtoreset{equation}{section}
\makeatother

\title
{
Periodicity of Grover walks on bipartite regular graphs with at most five distinct eigenvalues
}

\author{Sho Kubota\thanks{
Department of Applied Mathematics, Faculty of Engineering, Yokohama National University,
Hodogaya, Yokohama 240-8501, Japan. \texttt{kubota-sho-bp@ynu.ac.jp}
}
}

\date{}

\begin{document}
\maketitle
\begin{abstract}
We determine connected bipartite regular graphs with four distinct adjacency eigenvalues
that induce periodic Grover walks, and show that it is only $C_6$.
We also show that there are only three kinds of the second largest eigenvalues
of bipartite regular periodic graphs with five distinct eigenvalues.
Using walk-regularity,
we enumerate feasible spectra for such graphs.
\vspace{8pt} 
\end{abstract}

\section{Introduction}

Quantum walks are quantum analogues of classical random walks \cite{AAKV, ADZ, Gu}.
A great deal of research on quantum walks has been conducted in the last 20 years.
There is also a wide range of related fields.
In quantum information,
quantum walk models can be seen as generalizations of Grover's search algorithm \cite{Gr, P}.
More recently, quantum cryptography protocols based on quantum walks have been proposed \cite{PB, VRMPS}.

The subject of this paper is periodicity of quantum walks.
Periodicity has been studied as one of main problems of quantum walks,
and there are several previous studies.
Table~\ref{80} summarizes previous studies on periodicity of Grover walks on undirected graphs.
Other models have been studied in \cite{KSTY2019, Sa, SMA} for example.
Periodicity is a special case of state transfer problems.
The authors in \cite{KuSe} have applied periodicity to the study of perfect state transfer.
In context of quantum cryptography, 
periodicity of quantum walks can be a focus of attention \cite{PB}.

\begin{table}[ht]
  \centering
  \begin{tabular}{|c|c|}
\hline
Graphs & Ref. \\
\hline
\hline
Paths and Cycles & Trivial (or \cite{KSeYa}) \\ \hline
Complete graphs, complete bipartite graphs, SRGs & \cite{HKSS2017} \\ \hline
Generalized Bethe trees & \cite{KSTY2018} \\ \hline
Hamming graphs, Johnson graphs & \cite{Y2019} \\ \hline
Cycles (3-state) & \cite{KKKS} \\ \hline
Complete graphs with self loops & \cite{IMT} \\ \hline
\end{tabular}
 \caption{Previous works on periodicity of Grover walks on undirected graphs} \label{80}
\end{table}

In this paper,
we consider bipartite regular graphs with at most five distinct adjacency eigenvalues
that induce periodic Grover walks.
As we will see in the beginning of Section~\ref{1028-3},
the graphs with at most three distinct adjacency eigenvalues have been substantially studied before.
We will therefore study periodicity of the graphs with four or five distinct adjacency eigenvalues.
There are two main theorems.
See later sections for terminologies and definitions.
The first main result states that
if a bipartite regular graph with four or five distinct adjacency eigenvalues is periodic,
then the second largest eigenvalue can only take three different values.
The second main result is that
the only bipartite regular graph with four distinct adjacency eigenvalues
to induce periodic Grover walk is $C_6$:

\begin{thm}
{\it
Let $\G$ be a bipartite $k$-regular graph with the $A$-spectrum
$\{ [\pm k]^{1}, [\pm \theta]^{a}, [0]^{b} \}$,
where $a \geq 1$ and $b \geq 0$.
Then $\G$ is periodic if and only if
$k$ is even and $\theta \in \{ \frac{k}{2}, \frac{\sqrt{2}}{2}k, \frac{\sqrt{3}}{2}k \}$.
}
\end{thm}

\begin{thm}
{\it
Let $\G$ be a bipartite $k$-regular graph with the $A$-spectrum
$\{ [\pm k]^{1}, [\pm \theta]^{\frac{n}{2}-1}\}$,
where $n$ is the number of vertices of $\G$.
Then $\G$ is periodic if and only if $\G$ is isomorphic to the cycle graph $C_6$.
}
\end{thm}

This paper is organized as follows.
Section~\ref{1028-2} is mainly preparation.
Terms and facts related to spectral graph theory and Grover walks are introduced.
In Section~\ref{1028-3}, we derive the first main theorem.
Using rings of integers of quadratic fields,
we derive conditions on eigenvalues that periodic graphs have.
In Section~\ref{1028-4}, we derive the second main theorem.
In Section~\ref{1028-5}, we study bipartite regular graphs with five distinct adjacency eigenvalues.
Focusing on walk-regularity, we enumerate feasible spectra that periodic graphs have.
In Section~\ref{1028-6}, we summarize the results and discuss open problems.

\section{Preliminaries} \label{1028-2}

See \cite{GR} for basic terminologies related to graphs.
Throughout this paper,
we assume that all graphs are finite, simple, and connected whether we mention or not.
Let $M$ be a square matrix.
We call the multiset of eigenvalues of $M$ the {\it spectrum} or the {\it $M$-spectrum},
and denote it by $\Spec(M)$.
For example, when the $M$-spectrum is $\{ 4,0,0,0,-2,-2\}$,
we describe the multiplicities by superscript, as in $\{ [4]^1, [0]^3, [-2]^2 \}$.
The matrices $I_n$ and $J_n$ denote the identity matrix and the all-ones matrix of size $n$, respectively.
The subscripts can be omitted if the sizes of these matrices are clear in context.

\subsection{Graphs and their spectra}

The {\it adjacency matrix} $A=A(\G) \in \MB{C}^{V \times V}$ of a graph $\G = (V, E)$ is defined by
\[ A_{x,y} = \begin{cases}
1 \qquad &\text{if $\{ x,y \} \in E$,} \\
0 \qquad &\text{otherwise.}
\end{cases} \]
We call eigenvalues of $A(\G)$ {\it adjacency eigenvalues}.

\begin{pro}[Proposition~3.3.1 in \cite{BH}] \label{0925-1}
{\it
Let $\G$ be a $k$-regular graph with the adjacency eigenvalues
$k=\lambda_1 \geq \lambda_2 \geq \dots \geq \lambda_n$,
where $n$ is the number of vertices.
Then we have $\sum_{i=1}^n \lambda_i^2 = nk$.
}
\end{pro}

A graph $\G = (V, E)$ is said to be {\it bipartite}
if its vertex set can be partitioned into two parts $V_1$ and $V_2$ such that
for any edge one end is in $V_1$ and the other end is in $V_2$.
The two parts $V_1, V_2$ are called {\it partite sets} of $V$.
It is well-known that the $A$-spectrum of a bipartite graph is symmetric about the origin and vice versa.
See Theorem~8.8.2 in \cite{GR} for details.
In this paper,
we are mainly concerned with (connected) bipartite regular graphs
with four or five distinct adjacency eigenvalues.
The $A$-spectrum of such a graph is usually denoted by $\{ [k]^{1}, [\theta]^{a}, [0]^b, [-\theta]^{a}, [-k]^1 \}$,
but we will simply write it as $\{ [\pm k]^{1}, [\pm \theta]^{a}, [0]^b \}$.
The following is a theorem due to Hoffman,
often used in studies of graphs with few distinct adjacency eigenvalues.

\begin{pro}[Theorem~1 and its proof in \cite{H}] \label{0930-2}
{\it
Let $\G$ be a connected $k$-regular graph with $n$ vertices,
and let the distinct adjacency eigenvalues be $k > \lambda_2 > \dots > \lambda_s$.
Then we have
\[ q(A(\G)) = \frac{q(k)}{n}J_n, \]
where $q(x) = \prod_{i=2}^s (x-\lambda_i)$.
}
\end{pro}

\subsection{Grover walks} \label{1001-1}

Let $\G$ be a  graph. 
Define $\MC{A} = \MC{A}(\G)=\{ (x, y), (y, x) \mid \{x, y\} \in E(\G) \}$.
The origin $x$ and terminus $y$ of $a=(x, y) \in \MC{A}$ are denoted by $o(a), t(a)$, respectively.
We write the inverse arc of $a$ as $a^{-1}$.

We introduce several matrices on Grover walks for a graph $\G$.
The {\it boundary matrix} $d = d(\G) \in \MB{C}^{V \times \MC{A}}$ is defined by
$d_{x,a} = \frac{1}{\sqrt{\deg x}} \delta_{x, t(a)}$,
where $\delta_{a,b}$ is the Kronecker delta.
The {\it shift matrix} $S = S(\G) \in \MB{C}^{\MC{A} \times \MC{A}}$
is defined by $S_{a, b} = \delta_{a,b^{-1}}$.
Define the {\it time evolution matrix} $U = U(\G) \in \MB{C}^{\MC{A} \times \MC{A}}$
by $U = S(2d^*d-I)$.
Quantum walks defined by $U$ is called {\it Grover walks}.
The {\it discriminant} $T=T(\G) \in \MB{C}^{V \times V}$ is defined by $T = dSd^*$.
Henceforth we will consider not only the spectrum of the adjacency matrix,
but also the ones of the discriminant and the time evolution matrix.
Define $\Spec_{A}(\G) = \Spec(A(\G))$ for a graph $\G$.
$\Spec_{T}(\G)$ and $\Spec_{U}(\G)$ are defined in the same way.

\begin{lem} \label{0925-2}
{\it
Let $\G$ be a $k$-regular graph,
and let $A$ and $T$ be the adjacency matrix and the discriminant of $\G$, respectively.
Then we have $T = \frac{1}{k}A$.
Therefore, the absolute values of eigenvalues of $T$ does not exceed $1$.
}
\end{lem}

We omit a proof.
See Section~3 in \cite{KST} for more general claim and its proof.
Relationship between $U$-spectra and $T$-spectra has been studied
not only in the Grover walks but also in more general models \cite{HKSS2014, KSSS, KSY}.
We cite the result in \cite{HKSS2014},
but the statement is slightly modified to fit our setting.

\begin{thm}[\cite{HKSS2014}] \label{0925-3}
{\it
Let $\G = (V, E)$ be a connected graph.
Then we have
\[ \Spec_{U}(\G) = \{ e^{\pm i \arccos \lambda} \mid \lambda \in \Spec_{T}(\G) \} \cup \{ 1 \}^{M_1} \cup \{-1\}^{M_{-1}}, \]
where $M_{1} = |E| - |V| + 1$ and $M_{-1} = |E| - |V| + \dim \Ker (T + I)$.
}
\end{thm}

By Lemma~\ref{0925-2} and Theorem~\ref{0925-3},
the $U$-spectrum is obtained from the $A$-spectrum via the discriminant when a graph is regular.

Let $U$ be the time evolution matrix of a graph $\G$.
We say that $\G$ is {\it periodic} if there exists a positive integer $\tau$ such that $U^{\tau} = I$.
When a graph $\G$ is periodic,
the positive integer $\min \{ \tau \in \MB{N} \mid U^{\tau} = I \}$ is called the {\it period}.
As we see immediately, periodicity is determined by $U$-spectrum.

\begin{lem}[Lemma~5.3 in \cite{KSeYa}] \label{1110-1}
{\it
Let $U$ be the time evolution matrix of a graph $\G$.
Then, we have
\[
\{ \tau \in \MB{N} \mid U^{\tau} = I \} =
\{ \tau \in \MB{N} \mid \lambda^{\tau} = 1 \text{ \emph{for any} $\lambda \in \Spec_U(\G)$} \}.
\]
In particular, $\G$ is periodic if and only if
there exists a positive integer $\tau$ such that $\lambda^{\tau} = 1$
for any eigenvalue $\lambda$ of $U$.
}
\end{lem}

\section{Bipartite regular graphs with at most five distinct adjacency eigenvalues} \label{1028-3}

In this section,
we derive a general fact on bipartite regular periodic graphs with four or five distinct adjacency eigenvalues.
Note that the bipartite regular graph with two distinct adjacency eigenvalues
is the complete graph $K_2$, which is known to be periodic \cite{HKSS2017}.
Bipartite regular graphs with three distinct adjacency eigenvalues
have the $A$-spectra of the form $\{[k]^1, [0]^{n-2}, [-k]^{1} \}$.
It is well-known that connected regular graphs with three distinct adjacency eigenvalues are strongly regular.
Recovering the parameters from the eigenvalues,
we see that their complements are disconnected.
Disconnected strongly regular graphs are disjoint unions of the complete graphs with the same size
(See Lemma~10.1.1 in \cite{GR}).
This implies that the graphs are the complete bipartite graphs $K_{k,k}$,
which are periodic shown in \cite{HKSS2017}.
In the end, we are concerned with bipartite regular graphs with four or five distinct adjacency eigenvalues.

Let $\G$ be a bipartite $k$-regular graph with $n$ vertices and the $A$-spectrum $\{ [\pm k]^{1}, [\pm \theta]^{a}, [0]^{b} \}$,
where $a \geq 1$ and $b \geq 0$.
Proposition~\ref{0925-1} derives $2k^2 + 2a \theta^2 = nk$.
We may assume that $\theta > 0$ without loss of generality,
and we have
\begin{equation} \label{0716-2}
\theta = \sqrt{\frac{nk-2k^2}{2a}}.
\end{equation}
On the other hand,
let the partite set of the vertex set be $V_1$ and $V_2$,
then we obtain
\begin{equation} \label{0717-1}
|V_1| = |V_2|.
\end{equation} 
Indeed, since $\G$ is bipartite,
we can display as
\[ A(\G) = \MM{O}{N}{N^{\top}}{O} \]
for some $p \times q$ matrix $N$.
The graph $\G$ is $k$-regular,
so $N \BM{1}_q = k \BM{1}_p$ and $N^{\top} \BM{1}_p = k \BM{1}_q$,
where $\BM{1}_r$ denotes the all-ones vector of size $r$.
We have $kp = \BM{1}_p^{\top} (N \BM{1}_q) = (\BM{1}_p^{\top} N) \BM{1}_q = kq$, i.e., $p = q$.
This implies $|V_1| = |V_2|$.

A complex number $\alpha$ is said to be an {\it algebraic integer}
if there exists a monic polynomial $p(x) \in \MB{Z}[x]$ such that $p(\alpha) = 0$.
Let $\Omega$ denote the set of algebraic integers.
Note that since the characteristic polynomial of the adjacency matrix
is a monic polynomial with integer coefficients,
the adjacency eigenvalues of a graph $\G$ are algebraic integers,
i.e.,
\begin{equation} \label{1110-2}
\Spec_{A}(\G) \subset \Omega.
\end{equation}
It is well-known that $\Omega$ is a subring of $\MB{C}$, and $\Omega \cap \MB{Q} = \MB{Z}$.
See \cite{J, L} for algebraic integers.
In addition, we will use integral bases of quadratic fields.
A positive integer $m >1$ is said to be {\it square-free}
if it is not divisible by $p^2$ for any prime number $p$. 

\begin{pro}[Proposition~2.34 in \cite{J}] \label{0930-1}
{\it
Let $m>1$ be a square-free integer.
Then
\[
\Omega \cap \MB{Q}(\sqrt{m}) = \begin{cases}
\{ p + q\sqrt{m} \mid p,q \in \MB{Z} \} \quad &\text{if $m \equiv 2,3 \pmod 4$}, \\
\{ p + \frac{1+\sqrt{m}}{2} q \mid p,q \in \MB{Z} \} \quad &\text{if $m \equiv 1 \pmod 4$}.
\end{cases} 
\]
}
\end{pro}
Note that $m$ is not square-free when $m \equiv 0 \pmod 4$.

\begin{lem} \label{0716-1}
{\it
Let $\G$ be a periodic graph.
If $\lambda \in \Spec_{T}(\G)$, then $2\lambda \in \Omega$.
}
\end{lem}

\begin{proof}
By Lemma~\ref{1110-1}, there exists a positive integer $\tau$
such that $\Lambda^{\tau} = 1$ for any $\Lambda \in \Spec_{U}(\G)$.
The eigenvalue $\Lambda$ is a root of the monic polynomial $x^{\tau} - 1 \in \MB{Z}[x]$,
so it is an algebraic integer.
Let $\lambda \in \Spec_{T}(\G)$.
Theorem~\ref{0925-3} derives $e^{\pm i \arccos \lambda} \in \Spec_{U}(\G)$.
Since $\Omega$ is a ring,
we have $2\lambda  = e^{i \arccos \lambda} + e^{- i \arccos \lambda} \in \Omega$.
\end{proof}

\begin{thm} \label{0719-1}
{\it
Let $\G$ be a bipartite $k$-regular graph with the $A$-spectrum
$\{ [\pm k]^{1}, [\pm \theta]^{a}, [0]^{b} \}$,
where $a \geq 1$ and $b \geq 0$.
Then $\G$ is periodic if and only if
$k$ is even and $\theta \in \{ \frac{k}{2}, \frac{\sqrt{2}}{2}k, \frac{\sqrt{3}}{2}k \}$.
}
\end{thm}

\begin{proof}
First, we show the sufficient condition for $\theta = \frac{k}{2}$.
The other cases are shown in the same way.
By Lemma~\ref{0925-2}, we have $\Spec_{T}(\G) = \{[\pm 1]^1, [\pm \frac{1}{2}]^a, [0]^b\}$.
Theorem~\ref{0925-3} implies
$\Spec_{U}(\G) \subset \{ \pm 1, e^{\pm \frac{\pi}{3}i}, e^{\pm \frac{\pi}{2}i} \}$,
where the multiplicities are ignored.
Thus, there exists a positive integer $\tau$ such that
$\lambda^{\tau} = 1$ for any $\lambda \in \Spec_{U}(\G)$,
and hence $\G$ is periodic by Lemma~\ref{1110-1}.

We show the necessary condition.
We may assume that $\theta > 0$ without loss of generality.
We have $\frac{\theta}{k} \in \Spec_{T}(\G)$ since $\G$ is $k$-regular.
Lemma~\ref{0716-1} derives $\frac{2 \theta}{k} \in \Omega$.

Consider the case $\theta \in \MB{Q}$.
We have $\frac{2 \theta}{k} \in \Omega \cap \MB{Q} = \MB{Z}$.
Lemma~\ref{0925-2} implies 
$0 < \frac{\theta}{k} < 1$.
Since $\frac{2 \theta}{k}$ is an integer, 
we have $\frac{2\theta}{k} = 1$, i.e., $\theta = \frac{k}{2}$.
From (\ref{1110-2}), we have 
$\theta \in \Omega \cap \MB{Q} = \MB{Z}$, and hence $k$ is even.

Next, we consider the case $\theta \not\in \MB{Q}$.
By (\ref{0716-2}), there exist a square-free integer $m > 1$ and $r \in \MB{Q}$
such that $\theta = r \sqrt{m}$.
Thus, we have
\begin{equation} \label{1126-1}
\frac{2 \theta}{k} = \frac{2r}{k} \sqrt{m} \in \Omega \cap \MB{Q}(\sqrt{m}).
\end{equation}
We claim that $m \not\equiv 1 \pmod 4$.
Suppose $m \equiv 1 \pmod 4$.
By Proposition~\ref{0930-1}, there exists $p,q \in \MB{Z}$ such that
\[ \frac{2r}{k}\sqrt{m} = p + \frac{1+\sqrt{m}}{2}q. \]
Since $1, \sqrt{m}$ are linearly independent over $\MB{Q}$,
we have
\begin{equation} \label{0716-4}
p+\frac{q}{2} = 0
\end{equation}
and
\begin{equation} \label{0716-3}
\frac{2r}{k} = \frac{q}{2}.
\end{equation}
Since $0 < \frac{2 \theta}{k} < 2$,
Equalities~(\ref{1126-1}) and~(\ref{0716-3}) imply
$0 < q\sqrt{m} < 4$.
Since $m \geq 5$, we have $0 < q < \frac{4}{\sqrt{m}} < 2$,
and hence $q=1$. 
Equality~(\ref{0716-4}) shows that $p = -\frac{1}{2}$, which contradicts to $p \in \MB{Z}$.
Now, we have $m \equiv 2,3 \pmod 4$.
By Proposition~\ref{0930-1}, there exists $p,q \in \MB{Z}$ such that $\frac{2r}{k}\sqrt{m} = p + q\sqrt{m}$.
Since $1, \sqrt{m}$ are linearly independent over $\MB{Q}$,
we have $p = 0$ and
\begin{equation} \label{1126-2}
q = \frac{2r}{k}.
\end{equation}
Since $0 < \frac{2 \theta}{k} < 2$,
Equalities~(\ref{1126-1}) and~(\ref{1126-2}) imply $0 < q\sqrt{m} < 2$.
Since $m \geq 2$, we have $0 < q < \frac{2}{\sqrt{m}} < 2$.
This implies $q=1$, and hence $(q,m) = (1,2), (1,3)$.
Therefore, we have $\frac{2\theta}{k} = \frac{2r}{k}\sqrt{m} = q\sqrt{m} \in \{\sqrt{2}, \sqrt{3}\}$,
i.e., $\theta \in \{\frac{\sqrt{2}}{2}k, \frac{\sqrt{3}}{2}k \}$.
The eigenvalue $\theta$ is also in $\Omega \cap \MB{Q}(\sqrt{m})$,
so $k$ must be even from Proposition~\ref{0930-1}. 
\end{proof}

\section{Four distinct adjacency eigenvalues} \label{1028-4}
In this section,
we consider a bipartite regular graph $\G$ with four distinct adjacency eigenvalues.
As mentioned in Proposition~15.1.3 in \cite{BH},
it is known that such a graph is the incidence graph of a symmetric 2-design.
Therefore there are a very large number of connected bipartite regular graphs with four distinct adjacency eigenvalues, of which we will show that only $C_6$ is periodic.
Let the $A$-spectrum of $\G$ be $\{ [\pm k]^{1}, [\pm \theta]^{a} \}$, where $a \geq 1$.
Let $n$ be the number of vertices of $\G$.
Then we have $a = \frac{n}{2} - 1$.
From the previous section,
if $\G$ is periodic, then $\theta \in \{ \frac{k}{2}, \frac{\sqrt{2}}{2}k, \frac{\sqrt{3}}{2}k \}$.

\begin{thm}
{\it
Let $\G$ be a bipartite $k$-regular graph with the $A$-spectrum
$\{ [\pm k]^{1}, [\pm \theta]^{\frac{n}{2} - 1}\}$,
where $n$ is the number of vertices of $\G$.
Then $\G$ is periodic if and only if $\G$ is isomorphic to the cycle graph $C_6$.
}
\end{thm}

\begin{proof}
It is well-known that the cycle graph $C_6$ has the $A$-spectrum $\{ [\pm 2]^{1}, [\pm 1]^{2} \}$,
which is periodic.
We show that the assumed periodic graph is determined to $C_6$.
Since $\G$ is bipartite,
its adjacency matrix is displayed as
\[ A = A(\G) = \MM{O}{N}{N^{\top}}{O} \]
for some matrix $N$.
This matrix $N$ is square matrix of size $\frac{n}{2}$ from (\ref{0717-1}).
We have
\begin{align*}
\{ [k^2]^{2}, [\theta^2]^{n-2} \} &=
\Spec(A^2) \\
&= \Spec(NN^{\top}) \cup \Spec(N^{\top} N) \\
&= \Spec(NN^{\top}) \cup \Spec(NN^{\top}). 
\end{align*}
This implies $\Spec(NN^{\top}) = \{ [k^2]^{1}, [\theta^2]^{\frac{n}{2} - 1} \}$,
that is,
\begin{equation}
\Spec(NN^{\top} - \theta^2 I) = \{ [k^2 - \theta^2]^{1}, [0]^{\frac{n}{2} - 1} \}.
\end{equation}
We take an eigenvector $\sqrt{\frac{2}{n}} \BM{1}_{\frac{n}{2}}$ of norm $1$
associated to the eigenvalue $k^2 - \theta^2$,
and eigenvectors $x_1, \dots, x_{\frac{n}{2} - 1}$ associated to the eigenvalue $0$
such that $x_{i}^{\top}x_j = \delta_{i,j}$.
The matrix $P=[\sqrt{\frac{2}{n}} \BM{1}_{\frac{n}{2}}, x_1, \dots, x_{\frac{n}{2} - 1}]$
is an orthogonal matrix and diagonalizes $NN^{\top} - \theta^2 I$.
Thus we have
\[ NN^{\top} - \theta^2 I = P \diag(k^2 - \theta^2, 0, \dots, 0) P^{\top}
= \frac{2(k^2 - \theta^2)}{n} J_{\frac{n}{2}}, \]
that is,
\begin{equation} \label{0717-2}
NN^{\top} = \theta^2 I + \frac{2(k^2 - \theta^2)}{n} J_{\frac{n}{2}}
\end{equation}
Since $\G$ is $k$-regular and $N$ is $\{0,1\}$-matrix, we have
\[(NN^{\top})_{x,x} = \sum_{z \in V(\G)} N_{x,z} N_{x,z} = \sum_{z \in V(\G)} N_{x,z} = k \]
for any vertex $x \in V(\G)$.
The $(x,x)$ entries of both sides of Equality~(\ref{0717-2}) are
$k = \theta^2 + \frac{2(k^2 - \theta^2)}{n}$.
In particular, we have 
\begin{equation} \label{0719-2}
k > \theta^2.
\end{equation}
Suppose $\theta \in \{\frac{\sqrt{2}}{2}k, \frac{\sqrt{3}}{2}k \}$.
Since $\theta \geq \frac{\sqrt{2}}{2}k$,
Inequality~(\ref{0719-2}) derives  $k < 2$.
However, $k$ is even by Theorem~\ref{0719-1}.
This is impossible.
Thus $\theta = \frac{k}{2}$.
By (\ref{0719-2}) again, we have $k < 4$, and hence $k=2$ because $k$ is even.
The $A$-spectrum of $\G$ is determined to be $\{ [\pm 2]^{1}, [\pm 1]^{\frac{n}{2} - 1} \}$.
Equality~(\ref{0716-2}) implies $n=6$.
Since $\G$ is connected,
the graph is determined to be the cycle graph $C_6$.
\end{proof}

\section{Five distinct adjacency eigenvalues} \label{1028-5}

In contrast to bipartite regular graphs with four distinct adjacency eigenvalues,
we can construct a large number of periodic graphs with five distinct adjacency eigenvalues.
On the other hand, of the feasible spectra,
there are many graphs whose existence is unknown.
Using tools of spectral graph theory,
we discuss feasible periodic graphs with five distinct adjacency eigenvalues.

\subsection{Construction}

Let $\G$ be a graph.
We write $\G \otimes J_m$ as the graph defined by
$A(\G \otimes J_m) = A(\G) \otimes J_m$.
If $m=1$, we simply regard $\G \otimes J_1$ as $\G$.
If a graph $\G$ is $k$-regular and has the $A$-spectrum
$\{ [k]^{1}, [\lambda_1]^{f_1}, \dots, [\lambda_s]^{f_s} \}$,
then the graph $\G \otimes J_m$ is $km$-regular, and
\begin{align*}
\Spec_{A}(\G \otimes J_m)
&= \{ [km]^{1}, [\lambda_1 m]^{f_1}, \dots, [\lambda_s m]^{f_s}, [0]^{nm-n} \}, \\
\Spec_{T}(\G \otimes J_m)
&= \left\{ [1]^{1}, \left[ \frac{\lambda_1}{k} \right]^{f_1}, \dots, \left[ \frac{\lambda_s}{k} \right]^{f_s}, [0]^{nm-n} \right \},
\end{align*}
where $n$ is the number of vertices of $\G$.
In particular, if $A(\G)$ has the eigenvalue 0,
then $T(\G)$ and $T(\G \otimes J_m)$ have the same spectrum except for the multiplicities.
From this observation,
it follows that  $C_6 \otimes J_m$ and $C_8 \otimes J_m$ are also periodic graphs.
Note that the $U$-spectrum of $C_6 \otimes J_m$ ignoring its multiplicities is
$\{ \pm 1, \pm i, e^{\pm \frac{\pi}{3}i}, e^{\pm \frac{2\pi}{3}i} \}$ from Theorem~\ref{0925-3}.
Thus, the period is $12$.
Similarly, the period of $C_8 \otimes J_m$ is $8$.

Another construction method is to use graph products.
Let $\G$ and $\D$ be graphs with $n$ and $m$ vertices, respectively.
The {\it Cartesian product} $\G \sikaku \D$ is the graph defined by
$A(\G \sikaku \D) = A(\G) \otimes I_m + I_n \otimes A(\D)$.
The {\it Kronecker product} $\G \otimes \D$ is the graph defined by
$A(\G \otimes \D) = A(\G) \otimes A(\D)$.
In particular, the graph $\G \otimes K_2$ is called the {\it bipartite double} of $\G$.
The $A$-spectrum of $\G \otimes K_2$ is $\Spec_{A}(\G) \cup -\Spec_{A}(\G)$.
These graph products can also be used to construct periodic graphs.
For example, the line graph of the $3$-dimensional hypercube $Q_3$, i.e.,
$L(Q_3)$ has the $A$-spectrum $\{ [4]^1, [2]^3, [0]^3, [-2]^5 \}$.
Thus, its bipartite double $L(Q_3) \otimes K_2$ has the $A$-spectrum $\{ [\pm 4]^1, [\pm 2]^8, [0]^6\}$,
which is a connected bipartite regular periodic graph with five distinct adjacency eigenvalues.
For more information on graph products, see \cite{BH, K2017}.

\subsection{Feasible spectra}

From Theorem~\ref{0719-1},
if a bipartite $k$-regular graph $\G$ with five distinct adjacency eigenvalues is periodic,
the second largest eigenvalue can only take one of the three values, and $k$ must be even.
However, there are still infinite possibilities for the number of vertices and multiplicities of eigenvalues.
We need to narrow down candidates for $A$-spectra of graphs.

\begin{lem} \label{0726-1}
{\it
Let $\G$ be a bipartite $k$-regular graph with the $A$-spectrum
$\{ [\pm k]^{1}, [\pm \theta]^{a}, [0]^{b} \}$,
where $a, b \geq 1$.
If we display the adjacency matrix $A$ of $\G$ as
\[ A = \MM{O}{N}{N^{\top}}{O}, \]
then we have $NN^{\top}N = \theta^2 N + \frac{2k}{n}(k^2 - \theta^2) J_{\frac{n}{2}}$,
where $n$ is the number of vertices.
}
\end{lem}

\begin{proof}
The two vectors
\[ \alpha = \frac{1}{\sqrt{n}} \begin{bmatrix} \BM{1}_{\frac{n}{2}} \\ \BM{1}_{\frac{n}{2}} \end{bmatrix},
\qquad
\beta = \frac{1}{\sqrt{n}} \begin{bmatrix} \BM{1}_{\frac{n}{2}} \\ -\BM{1}_{\frac{n}{2}} \end{bmatrix}
 \]
are eigenvectors of $A$ associated to $k, -k$ with norm $1$, respectively.
Let $x_1, \dots, x_{n-2}$ be eigenvectors of $A$ associated to $\lambda \not\in \{k, -k\}$
such that $x_i^{\top}x_j = \delta_{i,j}$.
Define $M = (A^2 - \theta^2 I_n )A$.
Then $x_1, \dots, x_{n-2}$ are eigenvectors of $M$ associated to $0$.
Consider $Q = [ \alpha, \beta, x_1, \dots, x_{n-2} ]$.
The matrix $Q$ is an orthogonal matrix and diagonalizes $M$, so we have
\begin{align*}
M &= Q \diag \left( (k^2-\theta^2)k, -(k^2-\theta^2)k, 0 \dots, 0 \right) Q^{\top} \\
&= (k^2-\theta^2)k ( \alpha \alpha^{\top} - \beta \beta^{\top} ) \\
&= \frac{2k(k^2-\theta^2)}{n} \MM{O}{J_{\frac{n}{2}}}{J_{\frac{n}{2}}}{O}.
\end{align*}
On the other hand,
\begin{align*}
M &= \left( A^2 - \theta^2 I_n \right)A \\
&= \left( \MM{O}{N}{N^{\top}}{O}^2 - \theta^2 I_n \right) \MM{O}{N}{N^{\top}}{O} \\
&= \MM{O}{NN^{\top}N - \theta^2 N}{N^{\top}NN^{\top} - \theta^2N^{\top}}{O}.
\end{align*}
Comparing (1,2)-block, we obtain the statement.
\end{proof}

\begin{lem} \label{0923-3}
{\it
With the above notation,
we have $n \leq 2k(k^2 - \theta^2)$.
}
\end{lem}
\begin{proof}
Pick $x \in V(\G)$.
Let $\G_j(x)$ be the set of vertices at distance $j$ from $x$.
Display as $\G_2(x) = \{ y_1, \dots, y_t \}$.
Suppose that $\G_1(x) = \G_1(y_i)$ for any $y_i \in \G_2(x)$.
Since $\G$ is connected,
$\G$ is isomorphic to the complete bipartite graph $K_{k,k}$.
However, this contradicts the assumption of the $A$-spectrum of $\G$.
Thus, there exists $y_i \in \G_2(x)$ such that $\G_1(y_i) \neq \G_1(x)$.
Since $\G$ is $k$-regular, neither $\G_1(y_i)$ nor $\G_1(x)$ is included in the other.
Thus, we have $\G_1(y_i) \setminus \G_1(x) \neq \emptyset$.
Pick $z \in \G_1(y_i) \setminus \G_1(x)$.
The pair of vertices $x,z$ satisfies $N_{x,z} = 0$ and $(NN^{\top}N)_{x,z} \geq 1$.
By Lemma~\ref{0726-1},
\[ 1 \leq (NN^{\top}N )_{x,z} = \left( \theta^2 N + \frac{2k}{n}(k^2 - \theta^2) J_{\frac{n}{2}} \right)_{x,z}
= \frac{2k}{n}(k^2 - \theta^2). \]
We have the statement.
\end{proof}

A graph $\G$ with the adjacency matrix $A$ is said to be {\it walk-regular}
if $(A^r)_{x,x}$ is independent of the choice of $x$ for each positive integer $r$.
Walk-regular graphs are regular since $\deg x = (A^2)_{x,x}$ is a constant.
See \cite{GM} for more information on walk-regular graphs.
Walk-regularity provides conditions for the existence of graphs.
Let the constant $(A^r)_{x,x}$ be $c_r$ for a positive integer $r$,
and let the adjacency eigenvalues be $k = \lambda_1 \geq \lambda_2 \geq \dots \geq \lambda_n$.
We have
\[ \sum_{i = 1}^n \lambda_i^r = \tr (A^r) = \sum_{x \in V(\G)} (A^r)_{x,x} = n c_r, \]
where $n$ is the number of vertices.
The constant $c_r$ is the number of closed walks, so
\[ c_r = \frac{1}{n} \sum_{i = 1}^n \lambda_i^r \]
is a non-negative integer.
This condition restricts feasible spectra of walk-regular graphs.
Van Dam pointed out that regular graphs with four adjacency eigenvalues are walk-regular \cite{vD}.
Koledin and Stani\'{c} showed that
regular bipartite graphs with three distinct non-negative adjacency eigenvalues are walk-regular \cite{KS}.
Although a proof is similar, the graphs we consider are also walk-regular.

\begin{lem} \label{0923-4}
{\it
Let $\G$ be a bipartite $k$-regular graph with the $A$-spectrum
$\{ [\pm k]^{1}, [\pm \theta]^{a}, [0]^{b} \}$,
where $a, b \geq 1$.
Then $\G$ is walk-regular.
}
\end{lem}

\begin{proof}
Let $n$ be the number of vertices of $\G$,
and let $A$ be the adjacency matrix.
By Proposition~\ref{0930-2}, we have $q(A) = \frac{q(k)}{n}J_n$,
where $q(x) = x(x+k)(x^2-\theta^2)$.
In particular,
$A^4$ can be expressed as a linear combination of $A^3, A^2, A, I_n$, and $J_n$
with rational coefficients.
Thus for any positive integer $r$,
there exist $\alpha_1, \alpha_2, \alpha_3, \alpha_4, \alpha_5 \in \MB{Q}$ such that
$A^r = \alpha_1 A^3 + \alpha_2 A^2 + \alpha_3 A + \alpha_4 I_n + \alpha_5 J_n$.
On the other hand, $(A^3)_{x,x} = A_{x,x} = 0$ for any $x \in V(\G)$ since $\G$ is bipartite.
We have $(A^r)_{x,x} = k\alpha_2 + \alpha_4 + \alpha_5$,
and hence $\G$ is walk-regular.
\end{proof}

\begin{pro} \label{0923-5}
{\it
Let $\G$ be a bipartite $k$-regular graph with $n$ vertices.
Suppose that the $A$-spectrum is $\Spec_{A}(\G) = \{ [\pm k]^{1}, [\pm \theta]^{a}, [0]^{b} \}$,
where $a, b \geq 1$.
Then we have the following.
\begin{enumerate}[(i)]
\item $a = \frac{nk-2k^2}{2\theta^2}$ and $b = n-2 - \frac{nk-2k^2}{\theta^2}$.
In particular,
$\frac{nk-2k^2}{2\theta^2}$ and $n-2 - \frac{nk-2k^2}{\theta^2}$ are positive integers;
\item $\frac{2(k^2+\theta^2)}{k} \leq n \leq 2k(k^2-\theta^2)$; and
\item For any positive integer $r$,
we have $\frac{1}{n} (2 \cdot k^{2r} + (nk-2k^2)\theta^{2r-2}) \in \MB{Z}$.
\end{enumerate}
}
\end{pro}

\begin{proof}
(i)
Since the sum of the multiplicities is the number of vertices, we have
\begin{equation}\label{0923-1}
2a+b+2 = n.
\end{equation}
Proposition~\ref{0925-1} implies
\begin{equation}\label{0923-2}
2k^2 + 2a\theta^2 = nk,
\end{equation}
which leads to $a = \frac{nk-2k^2}{2\theta^2}$.
Equalities~(\ref{0923-1}) and (\ref{0923-2}) imply $b = n-2 - \frac{nk-2k^2}{\theta^2}$.
Since the multiplicities of eigenvalues are positive integers,
$\frac{nk-2k^2}{2\theta^2}$ and $n-2 - \frac{nk-2k^2}{\theta^2}$ are also positive integers.

(ii)
Since $a \geq 1$,
we have $\frac{nk-2k^2}{2\theta^2} \geq 1$, i.e., $n \geq \frac{2(k^2+\theta^2)}{k}$.
On the other hand,
Lemma~\ref{0923-3} derives $n \leq 2k(k^2 - \theta^2)$.
Thus, we have $\frac{2(k^2+\theta^2)}{k} \leq n \leq 2k(k^2-\theta^2)$.

(iii)
By Lemma~\ref{0923-4}, the graph $\G$ is walk-regular.
Thus, the constant $c_{2r}$ is a non-negative integer for any positive integer $r$.
We have $\frac{1}{n} (2 \cdot k^{2r} + (nk-2k^2)\theta^{2r-2}) \in \MB{Z}$.
\end{proof}

The above proposition narrows down the candidates of graphs.
Indeed, the condition~(ii) makes possibility of $n$ finite when $k$ is fixed.
The conditions~(i) and~(iii) further restrict possibilities for $n$ and $k$.
For convenience, we call graphs to pass Proposition~\ref{0923-5} {\it feasible periodic graphs}.
Table~\ref{0926-1}, Table~\ref{0926-2}, and Table~\ref{0926-3},
which are shown after references for reasons of space,
list feasible periodic graphs.
In these tables, there are many graphs whose existence is unknown.
The column ``Existence" gives one example that realizes the spectrum if such graphs exist.
See \cite{BCN} for the symbols of the graphs.
In addition,
the symbol ``$-$" in the tables denotes that
the spectrum passes Proposition~\ref{0923-5} but the non-existence of a graph is shown by other reasons. The reasons are briefly described in the rightmost column of the tables.
For example,
a graph with $A$-spectrum of the form $\{ [\pm 4]^{1}, [\pm 2]^{a}, [0]^b \}$ is a 4-regular integral graph.
Thus, the classification by Stevanovi{\'c} \cite{S} apply,
and hence there is no such graph with more than $32$ vertices.

\subsection{Quadrangles}

We next focus on the sum of the fourth power of the eigenvalues.
This contains the information of quadrangles in graphs.
Such geometric information gives us a slightly stronger condition for the existence of graphs.
The following is substantially pointed out also in \cite{vD, S2011},
but we give a proof for wider readers.

\begin{lem}
{\it
Let $\G$ be a walk-regular graph with the adjacency eigenvalues
$k=\lambda_1 \geq \lambda_2 \geq \dots \geq \lambda_n$.
Denote by $q$ the number of quadrangles in $\G$ and
by $q_x$ the number of quadrangles containing a vertex $x \in V(\G)$.
Then we have
\begin{enumerate}[(i)]
\item $\sum_{i=1}^n \lambda_i^4 = n(2k^2-k) + 8q$; and
\item The number $q_x$ is a constant independent of the choice of a vertex,
and we have $q_x = \frac{4q}{n}$.
\end{enumerate}
}
\end{lem}

\begin{proof}
Let the adjacency matrix of $\G$ be $A$.
First, we observe the $(x,x)$ entry of $A^4$ for a vertex $x$.
It represents the number of closed walks of length $4$ from $x$ to $x$.
See Figure~\ref{1021-1}.
There are four possibilities for a closed walk of length $4$.
Case (a) shows a closed walk $(x, y_1, x, y_1, x)$ where only one vertex $y_1$ appears besides $x$.
There are $k$ ways to get $y_1$, so there are $k$ walks of this type.
Case (b) shows a closed walk $(x, y_1, x, y_2, x)$ where two vertices $y_1, y_2$ are adjacent to $x$.
The total number getting $y_1, y_2$ is $k(k-1)$.
Thus, there are $k(k-1)$ walks of this type.
Case (c) shows a closed walk $(x, y_1, y_2, y_1, x)$
where $y_1$ is adjacent to $x$ but $y_2$ is not adjacent to $x$.
There are $k$ ways to take a vertex $y_1$,
and for each of them there are $k-1$ ways to take $y_2$.
Thus, there are $k(k-1)$ walks of this type.
Case (d) shows a closed walk that forms a quadrangle.
For each quadrangle, there are two closed walks $(x, y_1, y_2, y_3, x)$ and $(x, y_3, y_2, y_1, x)$.
Thus, we get $2q_x$ walks of this type.
By the four cases, we have
\begin{equation} \label{1007-1}
(A^4)_{x,x} = 2k^2 - k + 2q_x.
\end{equation}
Let the set of the quadrangles in $\G$ be $\MC{Q}$.
Count $|\{ (x,Q) \in V(\G) \times \MC{Q} \mid x \in Q \}|$ in two ways, and we obtain
\begin{equation} \label{1111-1}
\sum_{x \in V(\G)} q_x = 4q.
\end{equation}
Indeed,
\[
|\{ (x,Q) \in V(\G) \times \MC{Q} \mid x \in Q \}| = \sum_{x \in V(\G)} | \{ Q \in \MC{Q} \mid x \in Q \} |
= \sum_{x \in V(\G)} q_x,
\]
and we have
\[
|\{ (x,Q) \in V(\G) \times \MC{Q} \mid x \in Q \}| = \sum_{Q \in \MC{Q}} |\{ x \in V(\G) \mid x \in Q \}|
= \sum_{Q \in \MC{Q}} 4 = 4q.
\]
Thus,
\[
\sum_{i=1}^n \lambda_i^4 = \sum_{x \in V(\G)}(A^4)_{x,x}
= \sum_{x \in V(\G)}(2k^2 - k + 2q_x)
= n(2k^2-k) + 8q.
\]
Equality~(\ref{1007-1}) and walk-regularity derive $c_4 = (A^4)_{x,x} = 2k^2 - k + 2q_x$,
so $q_x$ is also a constant.
Thus, Equality~(\ref{1111-1}) implies $4q = \sum_{x \in V(\G)} q_x = nq_x$,
that is, $q_x = \frac{4q}{n}$.
\end{proof}

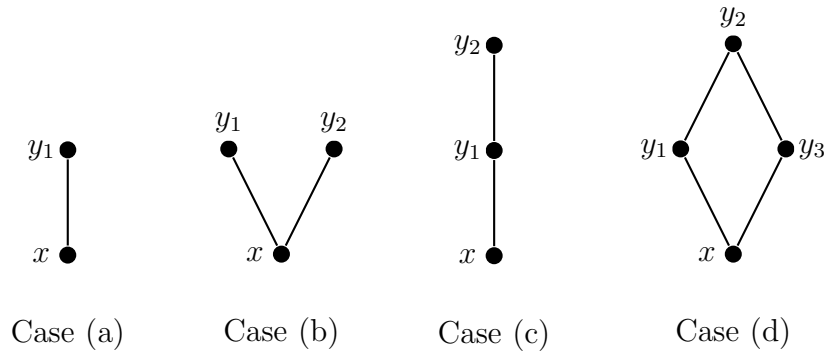
\begin{figure}[ht]
\begin{center}
\begin{tikzpicture}
[scale = 0.7,
line width = 0.8pt,
v/.style = {circle, fill = black, inner sep = 0.8mm},u/.style = {circle, fill = white, inner sep = 0.1mm}]
  \node[u] (100) at (0, -1.5) {Case (a)};
  \node[v] (1) at (0, 0) {};
  \node[u] (10) at (-0.5, 0) {$x$};
  \node[v] (2) at (0, 2) {};
  \node[u] (20) at (-0.5, 2) {$y_1$};
  \draw[-] (1) to (2);
\end{tikzpicture}
$\qquad$
\begin{tikzpicture}
[scale = 0.7,
line width = 0.8pt,
v/.style = {circle, fill = black, inner sep = 0.8mm},u/.style = {circle, fill = white, inner sep = 0.1mm}]
  \node[u] (100) at (0, -1.5) {Case (b)};
  \node[v] (1) at (0, 0) {};
  \node[u] (10) at (-0.5, 0) {$x$};
  \node[v] (2) at (-1, 2) {};
  \node[u] (20) at (-1, 2.5) {$y_1$};
  \node[v] (3) at (1, 2) {};
  \node[u] (30) at (1, 2.5) {$y_2$};
  \draw[-] (1) to (2);
  \draw[-] (1) to (3);
\end{tikzpicture}
$\qquad$
\begin{tikzpicture}
[scale = 0.7,
line width = 0.8pt,
v/.style = {circle, fill = black, inner sep = 0.8mm},u/.style = {circle, fill = white, inner sep = 0.1mm}]
  \node[u] (100) at (0, -1.5) {Case (c)};
  \node[v] (1) at (0, 0) {};
  \node[u] (10) at (-0.5, 0) {$x$};
  \node[v] (2) at (0, 2) {};
  \node[u] (20) at (-0.5, 2) {$y_1$};
  \node[v] (3) at (0, 4) {};
  \node[u] (30) at (-0.5, 4) {$y_2$};
  \draw[-] (1) to (2);
  \draw[-] (2) to (3);
\end{tikzpicture}
$\qquad$
\begin{tikzpicture}
[scale = 0.7,
line width = 0.8pt,
v/.style = {circle, fill = black, inner sep = 0.8mm},u/.style = {circle, fill = white, inner sep = 0.1mm}]
  \node[u] (100) at (0, -1.5) {Case (d)};
  \node[v] (1) at (0, 0) {};
  \node[u] (10) at (-0.5, 0) {$x$};
  \node[v] (2) at (-1, 2) {};
  \node[u] (20) at (-1.5, 2) {$y_1$};
  \node[v] (3) at (0, 4) {};
  \node[u] (30) at (0, 4.5) {$y_2$};
  \node[v] (4) at (1, 2) {};
  \node[u] (40) at (1.5, 2) {$y_3$};
  \draw[-] (1) to (2);
  \draw[-] (2) to (3);
  \draw[-] (3) to (4);
  \draw[-] (4) to (1);  
\end{tikzpicture}
\end{center}
\caption{Closed walks of length $4$} \label{1021-1}
\end{figure}

From the above lemma,
$q = \frac{1}{8}(\sum_{i=1}^n \lambda_i^4 - n(2k^2-k))$ and
$q_x = \frac{1}{2n}(\sum_{i=1}^n \lambda_i^4 - n(2k^2-k))$ must be non-negative integers.
This is another necessary condition
different from Proposition~\ref{0923-5} for the existence of graphs.
Indeed, several feasible periodic graphs in the tables are eliminated by this observation.

\begin{table}[h] {\footnotesize
  \centering
  \begin{tabular}{|c|c|c|c|c|}
\hline
$k$ & $n$ & Spectrum & Existence & Comment \\ \hline
4 & 12 & $\{ [\pm 4]^{1}, [\pm 2]^{2}, [0]^{6} \}$ & $C_6 \otimes J_2$ & \\
4 & 16 & $\{ [\pm 4]^{1}, [\pm 2]^{4}, [0]^{6} \}$ & $H(4,2)$ & \\ 
4 & 24 & $\{ [\pm 4]^{1}, [\pm 2]^{8}, [0]^{6} \}$ & $L(Q_3) \otimes K_2$ & \\
4 & 32 & $\{ [\pm 4]^{1}, [\pm 2]^{12}, [0]^{6} \}$ & $IG(AG(2,4) \setminus {\rm pc})$ & $q=0$, \cite{BCN, vDHKS} \\
4 & 48 & $\{ [\pm 4]^{1}, [\pm 2]^{20}, [0]^{6} \}$ & $-$ & \cite{S} \\ 
4 & 64 & $\{ [\pm 4]^{1}, [\pm 2]^{28}, [0]^{6} \}$ & $-$ & \cite{S} \\ 
4 & 96 & $\{ [\pm 4]^{1}, [\pm 2]^{44}, [0]^{6} \}$ & $-$ & \cite{S} \\ \hline 
6 & 18 & $\{ [\pm 6]^{1}, [\pm 3]^{2}, [0]^{12} \}$ & $C_6 \otimes J_3$ & \\
6 & 24 & $\{ [\pm 6]^{1}, [\pm 3]^{4}, [0]^{14} \}$ & $-$ & $q_x \not\in \MB{Z}$ \\
6 & 36 & $\{ [\pm 6]^{1}, [\pm 3]^{8}, [0]^{18} \}$ & ? & \\
6 & 54 & $\{ [\pm 6]^{1}, [\pm 3]^{14}, [0]^{24} \}$ & $H(3,3) \otimes K_2$ & \\
6 & 72 & $\{ [\pm 6]^{1}, [\pm 3]^{20}, [0]^{30} \}$ & $-$ & $q_x \not\in \MB{Z}$ \\
6 & 108 & $\{ [\pm 6]^{1}, [\pm 3]^{32}, [0]^{42} \}$ & ? & \\
6 & 162 & $\{ [\pm 6]^{1}, [\pm 3]^{50}, [0]^{60} \}$ & $IG(pg(5,5,2))$ & $q=0$, \cite{BCN, vDH} \\
6 & 216 & $\{ [\pm 6]^{1}, [\pm 3]^{68}, [0]^{78} \}$ & $-$ & $q<0$ \\
6 & 324 & $\{ [\pm 6]^{1}, [\pm 3]^{104}, [0]^{114} \}$ & $-$ & $q<0$ \\ \hline
8 & 24 & $\{ [\pm 8]^{1}, [\pm 4]^{2}, [0]^{18} \}$ & $C_6 \otimes J_4$ & \\
8 & 32 & $\{ [\pm 8]^{1}, [\pm 4]^{4}, [0]^{22} \}$ & $H(4,2) \otimes J_2$ & \\
8 & 48 & $\{ [\pm 8]^{1}, [\pm 4]^{8}, [0]^{30} \}$ & $L(Q_3) \otimes K_2 \otimes J_2$ & \\
8 & 64 & $\{ [\pm 8]^{1}, [\pm 4]^{12}, [0]^{38} \}$ & $K_{4,4} \sikaku K_{4,4}$ & \\
8 & 96 & $\{ [\pm 8]^{1}, [\pm 4]^{20}, [0]^{54} \}$ & ? & \\
8 & 128 & $\{ [\pm 8]^{1}, [\pm 4]^{28}, [0]^{70} \}$ & ? & \\
8 & 192 & $\{ [\pm 8]^{1}, [\pm 4]^{44}, [0]^{102} \}$ & ? & \\
8 & 256 & $\{ [\pm 8]^{1}, [\pm 4]^{60}, [0]^{134} \}$ & ? & \\
8 & 384 & $\{ [\pm 8]^{1}, [\pm 4]^{92}, [0]^{198} \}$ & ? & \\
8 & 512 & $\{ [\pm 8]^{1}, [\pm 4]^{124}, [0]^{262} \}$ & ? & \\
8 & 768 & $\{ [\pm 8]^{1}, [\pm 4]^{188}, [0]^{390} \}$ & ? & \\ \hline
10 & 30 & $\{ [\pm 10]^{1}, [\pm 5]^{2}, [0]^{24} \}$ & $C_6 \otimes J_5$ & \\
10 & 40 & $\{ [\pm 10]^{1}, [\pm 5]^{4}, [0]^{30} \}$ & $-$ & $q_x \not\in \MB{Z}$ \\
10 & 50 & $\{ [\pm 10]^{1}, [\pm 5]^{6}, [0]^{36} \}$ & ? & \\
10 & 60 & $\{ [\pm 10]^{1}, [\pm 5]^{8}, [0]^{42} \}$ & ? & \\
10 & 100 & $\{ [\pm 10]^{1}, [\pm 5]^{16}, [0]^{66} \}$ & ? & \\
10 & 120 & $\{ [\pm 10]^{1}, [\pm 5]^{20}, [0]^{78} \}$ & $-$ & $q_x \not\in \MB{Z}$ \\
10 & 150 & $\{ [\pm 10]^{1}, [\pm 5]^{26}, [0]^{96} \}$ & ? & \\
10 & 200 & $\{ [\pm 10]^{1}, [\pm 5]^{36}, [0]^{126} \}$ & $-$ & $q_x \not\in \MB{Z}$ \\
10 & 250 & $\{ [\pm 10]^{1}, [\pm 5]^{46}, [0]^{156} \}$ & ? & \\
10 & 300 & $\{ [\pm 10]^{1}, [\pm 5]^{56}, [0]^{186} \}$ & ? & \\
10 & 500 & $\{ [\pm 10]^{1}, [\pm 5]^{96}, [0]^{306} \}$ & ? & \\
10 & 600 & $\{ [\pm 10]^{1}, [\pm 5]^{116}, [0]^{366} \}$ & $-$ & $q_x \not\in \MB{Z}$ \\
10 & 750 & $\{ [\pm 10]^{1}, [\pm 5]^{146}, [0]^{456} \}$ & ? & \\
10 & 1000 & $\{ [\pm 10]^{1}, [\pm 5]^{196}, [0]^{606} \}$ & $-$ & $q_x \not\in \MB{Z}$ \\
10 & 1250 & $\{ [\pm 10]^{1}, [\pm 5]^{246}, [0]^{756} \}$ & ? & \\
10 & 1500 & $\{ [\pm 10]^{1}, [\pm 5]^{296}, [0]^{906} \}$ & ? & \\ \hline
$\vdots$&$\vdots$&&& \\
\end{tabular}
 \caption{Feasible periodic graphs whose $A$-spectra are the form
 $\{ [\pm k]^{1}, [\pm \frac{k}{2}]^{a}, [0]^{b} \} $.} \label{0926-1}
}
\end{table}

\begin{table}[h] 
  \centering
  \begin{tabular}{|c|c|c|c|c|}
\hline
$k$ & $n$ & Spectrum & Existence & Comment \\ \hline
2 & 8 & $\{ [\pm 2]^{1}, [\pm \sqrt{2}]^{2}, [0]^{2} \}$ & $C_8$ & \\ \hline
4 & 16 & $\{ [\pm 4]^{1}, [\pm 2\sqrt{2}]^{2}, [0]^{10} \}$ & $C_8 \otimes J_2$ & \\
4 & 32 & $\{ [\pm 4]^{1}, [\pm 2\sqrt{2}]^{6}, [0]^{18} \}$ & TD$_{1}(2,4) \otimes J_{2,1}$ & \cite{vDS} \\
4 & 64 & $\{ [\pm 4]^{1}, [\pm 2\sqrt{2}]^{14}, [0]^{34} \}$ & ? & \\ \hline
6 & 18 & $\{ [\pm 6]^{1}, [\pm 3\sqrt{2}]^{1}, [0]^{14} \}$ & $-$ & $q \not\in \MB{Z}$ \\
6 & 24 & $\{ [\pm 6]^{1}, [\pm 3\sqrt{2}]^{2}, [0]^{18} \}$ & $C_8 \otimes J_3$ & \\
6 & 36 & $\{ [\pm 6]^{1}, [\pm 3\sqrt{2}]^{4}, [0]^{26} \}$ & ? & \\
6 & 48 & $\{ [\pm 6]^{1}, [\pm 3\sqrt{2}]^{6}, [0]^{34} \}$ & $-$ & $q_x \not\in \MB{Z}$ \\
6 & 54 & $\{ [\pm 6]^{1}, [\pm 3\sqrt{2}]^{7}, [0]^{38} \}$ & $-$ & $q \not\in \MB{Z}$ \\
6 & 72 & $\{ [\pm 6]^{1}, [\pm 3\sqrt{2}]^{10}, [0]^{50} \}$ & ? & \\
6 & 108 & $\{ [\pm 6]^{1}, [\pm 3\sqrt{2}]^{16}, [0]^{74} \}$ & ? & \\
6 & 144 & $\{ [\pm 6]^{1}, [\pm 3\sqrt{2}]^{22}, [0]^{98} \}$ & $-$ & $q_x \not\in \MB{Z}$ \\
6 & 162 & $\{ [\pm 6]^{1}, [\pm 3\sqrt{2}]^{25}, [0]^{110} \}$ & $-$ & $q \not\in \MB{Z}$ \\
6 & 216 & $\{ [\pm 6]^{1}, [\pm 3\sqrt{2}]^{34}, [0]^{146} \}$ & ? & \\ \hline
8 & 32 & $\{ [\pm 8]^{1}, [\pm 4\sqrt{2}]^{2}, [0]^{26} \}$ & $C_8 \otimes J_4$ & \\
8 & 64 & $\{ [\pm 8]^{1}, [\pm 4\sqrt{2}]^{6}, [0]^{50} \}$ & TD$_{1}(2,4) \otimes J_{2,1} \otimes J_2$ & \cite{vDS} \\
8 & 128 & $\{ [\pm 8]^{1}, [\pm 4\sqrt{2}]^{14}, [0]^{98} \}$ & ? & \\
8 & 256 & $\{ [\pm 8]^{1}, [\pm 4\sqrt{2}]^{30}, [0]^{194} \}$ & ? & \\
8 & 512 & $\{ [\pm 8]^{1}, [\pm 4\sqrt{2}]^{62}, [0]^{386} \}$ & ? & \\ \hline
10 & 40 & $\{ [\pm 10]^{1}, [\pm 5\sqrt{2}]^{2}, [0]^{34} \}$ & $C_8 \otimes J_5$ & \\
10 & 50 & $\{ [\pm 10]^{1}, [\pm 5\sqrt{2}]^{3}, [0]^{42} \}$ & $-$ & $q \not\in \MB{Z}$ \\
10 & 80 & $\{ [\pm 10]^{1}, [\pm 5\sqrt{2}]^{6}, [0]^{66} \}$ & $-$ & $q_x \not\in \MB{Z}$ \\
10 & 100 & $\{ [\pm 10]^{1}, [\pm 5\sqrt{2}]^{8}, [0]^{82} \}$ & ? & \\
10 & 200 & $\{ [\pm 10]^{1}, [\pm 5\sqrt{2}]^{18}, [0]^{162} \}$ & ? & \\
10 & 250 & $\{ [\pm 10]^{1}, [\pm 5\sqrt{2}]^{23}, [0]^{202} \}$ & $-$ & $q \not\in \MB{Z}$ \\
10 & 400 & $\{ [\pm 10]^{1}, [\pm 5\sqrt{2}]^{38}, [0]^{322} \}$ & $-$ & $q_x \not\in \MB{Z}$ \\
10 & 500 & $\{ [\pm 10]^{1}, [\pm 5\sqrt{2}]^{48}, [0]^{402} \}$ & ? & \\
10 & 1000 & $\{ [\pm 10]^{1}, [\pm 5\sqrt{2}]^{98}, [0]^{802} \}$ & ? & \\ \hline
$\vdots$&$\vdots$&&& \\
\end{tabular}
 \caption{Feasible periodic graphs whose $A$-spectra are the form
 $\{ [\pm k]^{1}, [\pm \frac{\sqrt{2}}{2}k]^{a}, [0]^{b} \}$.} \label{0926-2}
\end{table}

\begin{table}[h] 
  \centering
  \begin{tabular}{|c|c|c|c|c|}
\hline
$k$ & $n$ & Spectrum & Existence & Comment \\ \hline
4 & 32 & $\{ [\pm 4]^{1}, [\pm 2\sqrt{3}]^{4}, [0]^{22} \}$ & $-$ & \cite{vDS} \\ \hline
8 & 64 & $\{ [\pm 8]^{1}, [\pm 4\sqrt{3}]^{4}, [0]^{54} \}$ & ? & \\
8 & 256 & $\{ [\pm 8]^{1}, [\pm 4\sqrt{3}]^{20}, [0]^{214} \}$ & ? & \\ \hline
10 & 50 & $\{ [\pm 10]^{1}, [\pm 5\sqrt{3}]^{2}, [0]^{44} \}$ & $-$ & $q \not\in \MB{Z}$ \\
10 & 200 & $\{ [\pm 10]^{1}, [\pm 5\sqrt{3}]^{12}, [0]^{174} \}$ & $-$ & $q \not\in \MB{Z}$ \\
10 & 500 & $\{ [\pm 10]^{1}, [\pm 5\sqrt{3}]^{32}, [0]^{434} \}$ & ? & \\ \hline
12 & 96 & $\{ [\pm 12]^{1}, [\pm 6\sqrt{3}]^{4}, [0]^{86} \}$ & ? & \\
12 & 384 & $\{ [\pm 12]^{1}, [\pm 6\sqrt{3}]^{20}, [0]^{342} \}$ & ? & \\ \hline
14 & 196 & $\{ [\pm 14]^{1}, [\pm 7\sqrt{3}]^{8}, [0]^{178} \}$ & ? & \\
14 & 1372 & $\{ [\pm 14]^{1}, [\pm 7\sqrt{3}]^{64}, [0]^{1242} \}$ & ? & \\ \hline
16 & 128 & $\{ [\pm 16]^{1}, [\pm 8\sqrt{3}]^{4}, [0]^{118} \}$ & ? & \\
16 & 512 & $\{ [\pm 16]^{1}, [\pm 8\sqrt{3}]^{20}, [0]^{470} \}$ & ? & \\
16 & 2048 & $\{ [\pm 16]^{1}, [\pm 8\sqrt{3}]^{84}, [0]^{1878} \}$ & ? & \\ \hline
20 & 100 & $\{ [\pm 20]^{1}, [\pm 10\sqrt{3}]^{2}, [0]^{94} \}$ & ? & \\
20 & 160 & $\{ [\pm 20]^{1}, [\pm 10\sqrt{3}]^{4}, [0]^{150} \}$ & ? & \\
20 & 250 & $\{ [\pm 20]^{1}, [\pm 10\sqrt{3}]^{7}, [0]^{234} \}$ & ? & \\
20 & 400 & $\{ [\pm 20]^{1}, [\pm 10\sqrt{3}]^{12}, [0]^{374} \}$ & ? & \\
20 & 640 & $\{ [\pm 20]^{1}, [\pm 10\sqrt{3}]^{20}, [0]^{598} \}$ & ? & \\
20 & 1000 & $\{ [\pm 20]^{1}, [\pm 10\sqrt{3}]^{32}, [0]^{934} \}$ & ? & \\
20 & 1600 & $\{ [\pm 20]^{1}, [\pm 10\sqrt{3}]^{52}, [0]^{1494} \}$ & ? & \\
20 & 2500 & $\{ [\pm 20]^{1}, [\pm 10\sqrt{3}]^{82}, [0]^{2334} \}$ & ? & \\
20 & 4000 & $\{ [\pm 20]^{1}, [\pm 10\sqrt{3}]^{132}, [0]^{3734} \}$ & ? & \\ \hline
$\vdots$&$\vdots$&&& \\
\end{tabular}
 \caption{Feasible periodic graphs whose $A$-spectra are the form
 $\{ [\pm k]^{1}, [\pm \frac{\sqrt{3}}{2}k]^{a}, [0]^{b} \}$.} \label{0926-3}
\end{table}

\section{Discussion and questions} \label{1028-6}

In this paper,
we completely determined bipartite regular graphs with four distinct adjacency eigenvalues
that induce periodic Grover walks, and show that it is only $C_6$.
In addition, we showed that there are only three kinds of the second largest eigenvalues
of bipartite regular periodic graphs with five distinct eigenvalues.
Focusing on certain conditions obtained from walk-regularity and the number of quadrangles,
we have narrowed down periodic graphs.
However, there are still many in the tables whose existence is unknown.
Table~\ref{0926-1} is directly related to the existence and classification of integral graphs,
which is a problem in spectral graph theory.
Indeed, the case $k=4$ in Table~\ref{0926-1} was completed by the results of Stevanovi{\'c} \cite{S}.
Similar problems for the case $k \geq 6$ remain to be solved:

\begin{question}
{\it
Fix an even number $k \geq 6$.
Classify the integral graphs whose $A$-spectra are the form $\{[\pm k]^1, [\pm \frac{k}{2}]^a, [0]^b\}$.
}
\end{question}

%

None of the graphs in Table~\ref{0926-3} have been found so far.
The existence of such graphs should also be investigated:

\begin{question}
{\it
Is there a graph whose A-spectrum is the form $\{ [\pm k]^{1}, [\pm \frac{\sqrt{3}}{2}k]^{a}, [0]^{b} \}$?
}
\end{question}

\section*{Acknowledgements}
S.K. is supported by JSPS KAKENHI (Grant No. 20J01175).

\end{document}